\documentclass[a4paper]{article}

\usepackage{amssymb} 
\usepackage{epstopdf} 
\usepackage{amsmath} 
\usepackage{amsthm} 
\usepackage{mathrsfs} 

\newtheorem{theorem}{Theorem}[section] 
\newtheorem*{mn_thm_1}{Theorem A}

\newtheorem{lemma}[theorem]{Lemma} 
 
\newtheorem{proposition}[theorem]{Proposition} 
\newtheorem{definition}[theorem]{Definition} 
 
\newtheorem{remark}[theorem]{Remark}

\DeclareMathOperator{\gk}{\mathscr{GK}-dim}

\DeclareMathOperator{\End}{End}

\title{A dichotomy for the Gelfand--Kirillov dimensions of simple modules over simple differential rings}

\author{Ashish Gupta\footnote{Corresponding author}\  and Arnab Dey Sarkar\footnote{Graduate student.}}

\date{}

\begin{document}
\maketitle

\begin{abstract}
The Gelfand--Kirillov dimension has gained importance since its introduction as an tool in the study of non-commutative infinite dimensional algebras and their modules. In this paper we show a dichotomy for the Gelfand--Kirillov dimension of simple modules over certain simple rings of differential operators. We thus answer a question of J.~C. McConnell in  Representations of solvable Lie algebras V. On the Gelʹfand-Kirillov dimension of simple modules. J. Algebra 76 (1982), no. 2, 489--493 concerning this question for a class of algebras that arise as simple homomorphic images of solvable lie algebras. We also determine the Gelfand--Kirillov dimension of an induced module. 
\end{abstract}

\section{Introduction}
The usefulness of the Gelfand-Kirllov dimension (GK dimension) as a tool in the investigation of non-commutative algebras is now well established.
This is due to the fact that it is relatively easier to compute and work with. It is an exact dimension function for the finitely generated modules over almost commutative algebras. Moreover, for a finitely generated module over an almost commutative algebra the GK dimension is given by the degree of the Hilbert-Samuel polynomial which has coefficients from $\mathbb Q$. Thus the GK dimension of such a module is a non-negative integer.  

It is an easy consequence of the Hilbert Nullstellensatz that each simple module over an affine commutative algebra has GK dimension equal to zero (e.g., \cite[Lemma 1.17]{MR:2001}). It was once believed that such a stability holds true for the GK dimension of simple modules over the Weyl algebras $A_n(\mathbb K)$ over a field $\mathbb K$ of characteristic zero. But in 1985 certain counterexamples were constructed by J.~T. Stafford and it is now known that $A_n(\mathbb C)$ has a simple module $S_j$ with GK dimension $j$ for each $n \le j \le 2n - 1$  (\cite{Ctno}).     

Our aim in this paper is to show that for certain simple rings a weak stability holds in that the GK dimensions of simple modules always lie in a two element set containing the two extreme values. Note that simple rings are not amenable to any methods that are based on the study of ideals. Thus studying the simple modules of such rings may be helpful in learning about their structure.

For example, it was shown by J.~C. McConnell in \cite{JcM:1982} that for a simple algebra $R_{n + 1}$ that occurs as an image of the universal enveloping algebra of a solvable lie algebra there exist simple modules having GK dimension as low as one and as high as $n$ where $n + 1$ denotes the GK dimension of the algebra $R_{n + 1}$.  McConnell's example is the ring defined as follows.
Let $k$ be an extension field of $\mathbb Q$ such that $\dim_{\mathbb Q}k = \infty$ and let $\lambda_i, i \in \{1,\cdots, n\}$ be linearly independent over $\mathbb Q$. Let $G_n$ be the free abelian group of rank $n$ having the basis $x_1, x_2, \cdots, x_n$ and set \[ K_n = k[x_1^{\pm 1}, x_2^{\pm 1}, \cdots, x_n^{\pm 1}]. \] It is the group algebra $kG_n$ of $G_n$ over $k$.    
Then $R_{n + 1}$ is the Ore extension $K_n[x, \delta]$, where $\delta$ is the derivation of $K_n$ defined by $\delta(x_j) = \lambda_jx_j$. The assumption that the scalars $\lambda_i$ are linearly independent over $\mathbb Q$ means that $R_{n +1}$ is a simple ring~\cite{JcM:1982}. By \cite[Theorem 2.1]{JcM:1982} it has GK dimension $n + 1$.

The question was asked whether there exist simple 
$R_{n + 1}$-modules having GK dimension $\alpha$ where $1 < \alpha < n$.

In this article we answer this question in the negative and show the following.

\begin{mn_thm_1}
Let $K_n$ be the laurent polynomial algebra \[ K_n : = k[x_1^{\pm 1}, \cdots, x_n^{\pm 1}]  \] over a field $k$ and let \[ R_{n + 1} :=  K_n[x, \delta], \]
where $\delta$ is a $k$-derivation of $K_n$ of the form 
\[\delta = \sum_{i = 1}^n g_i\frac{\partial}{\partial x_i}. \]
for $g_i \in kx_i + k$ 
such that $R_{n + 1}$ is a simple ring.
Let $M$ be a simple $R_{n + 1}$-module. Then $\gk(M) \in \{1, n\}$. \end{mn_thm_1}

\begin{remark}
Clearly, J.~C. McConnell's example is obtained as a special case by taking $g_i = \lambda_i x_i$. 
\end{remark}

For a finitely generated module $Q$ over an affine algebra $A$, the following inequalities hold with regard to the GK dimension 
\begin{equation}
\label{eq:bounds_4_gkdim}
0 \le \gk(Q) \le \gk(A). \end{equation}

We refer to \cite[Proposition 5.1]{KL:2000} for a proof of this fact.
In the case of the algebra $R_{n + 1}$ in Theorem A the possibility \[ \gk(M) = \gk(R_{n + 1} ) = n + 1 \] 
is easily ruled out (e.g., \cite[Proposition  2.3]{JcM:1982} for a simple module $M$. In the paper \cite{JcM:1982}, it was also shown that $\gk(M) = 0$ is not possible. In this case $M$ would necessarily be finite dimensional over $k$ implying that \[ \dim_k(\End_k(M )) < \infty. \]  
 $R_{n + 1}$ is a simple domain  and so $M$ is a faithful simple module and so the last inequality says that $R_{n + 1}$ is finite dimensional over $k$ which is clearly false. It follows that $1 \le \gk(M) \le n$.  

Other examples of simple rings where such a dichotomy holds for the GK dimensions of simple modules are known. An example is the skew-Laurent extension 
$S_{n+1}$ over $K_n$ defined as 
\[ S_{n + 1} = K_n[x, \sigma] \]
where $\sigma$ is the automorphism of $K_n$ defined by $\sigma(x_i) = \lambda_ix_i$, where $\lambda_j \in k^{\ast}: = k\setminus \{0\}$ are assumed to generate a subgroup of $k^{\ast}$ of the maximal possible rank $n$. The following theorem was proved in \cite{GA:2013}.  

\begin{theorem}
Let $M$ be a simple $S_{n+1}$-module. Then $\gk(M) \in \{1, n \}$. Moreover for each $j \in \{1, n\}$ there exists a simple $S_{n +1}$-module $M$ 
with \[ \gk(M) = j. \] 
\end{theorem}

Our approach as exposed in this article is applicable to a wider class of rings than considered in this paper and can be easily adapted to more general Ore extensions $K_n[x, \sigma, \delta]$ that are simple rings with suitable conditions on the derivation or automorphism involved. In this paper all modules will be right modules.  

\section{GK dimension, Ore localization and critical modules}
In this section we review some basic facts on the GK dimension of finitely generated modules over affine algebras and also the basic tools that we will be employing in our proofs. 

The GK dimension of an affine $k$-algebra $\mathbb A$ is defined as follows. Let a generating set $\{a_1, a_2, \cdots, a_s \}$ be given for $\mathbb A$. For convenience we may assume that $1 \in \{a_1, \cdots, a_s \}$. There is an increasing sequence of $k$-subspaces of $\mathbb A$ which is given by \[ \mathbb A_0: = k, \ \ \ \ \ \mathbb A_1: = \sum_{j = 1}^n ka_j \ \ \ \ \  \mathrm{and,} \ \ \ \ \ \mathbb A_n : = \mathbb A_1^n. \] This defines the \emph{standard finite dimensional filtration} of $\mathbb A$ with respect to the given choice of generators. The GK dimension of $\mathbb A$  measures the asymptotic growth rate of the corresponding sequence of dimensions (measured over $k$). 
\[ \gk(\mathbb A) := \lim_{n \to \infty} \sup \log_n(\dim_k \mathbb A_n). \] 
It turns out that this definition is independent of the choice of a generating set for $\mathbb A$.
This idea can be extended to a finitely generated $\mathbb A$-module $M$ as follows. Pick a finite dimensional generating subspace $M_0$ of $M$ and set $M_n = M_0\mathbb A_m$. We define  
\begin{equation}\label{defn_GK_dim_modules}
\gk(M) := \lim_{n \to \infty}\sup\log_n\dim_kM_n
\end{equation}
Again the GK dimension of $M$ doesn't depend on the choice of a generating set of $\mathbb A$ or the generating subspace of $M$. We refer the interested reader to \cite{KL:2000} for the proofs of these facts and an excellent treatment of the GK dimension.
The GK dimension is rather well-behaved on the finitely generated modules over affine algebras. 

\begin{proposition}\label{GK_dim_bsc_ppty_1}
Let $R$ be an affine algebra with a finitely generated module $M_R$. If $R'$ is an affine subalgebra of $R$ and $M'_{R}$ is a finitely generated $R'$-submodule of $M$ then 
\[ \gk(M'_{R'}) \le \gk(M_R) \]
\end{proposition}
\begin{proof}
See Lemma 1.13(ii) of \cite{MR:2001}.
\end{proof}

\begin{proposition}\label{GK_dim_bsc_ppty_2}
Let $\mathbb A$ be an affine commutative $k$-algebra. 
Let \[ 0 \rightarrow N \rightarrow M \rightarrow L \rightarrow 0 \]
be a sequence of $\mathbb A$-modules. Then \[ \gk(M) = \max\{\gk(N), \gk(M/N) \}. \]  
\end{proposition}
The last proposition is actually true for almost commutative algebras~\cite{KL:2000}.

\begin{proof}
See \cite[Proposition 5.1(ii)]{KL:2000}.
\end{proof}

Recall that a multiplicative subset $S$ of a ring $R$ is called a \emph{right Ore subset} if for any $s \in S$ and $r \in R$ the intersection $sR \cap rS$ is nonempty. A right Ore set in a domain $R$ is always a \emph{right denominator set} and $R$ has a right (Ore) localization with respect to $S$ denoted as $RS^{-1}$. We refer the interested reader to \cite[Chapter 8]{GW:1989} for the proofs of the above statements further details. Let $M$ be a right $R$-module and $X$ a right Ore subset in $R$. The set of all $X$-torsion elements, namely,  
\[ \mathrm{T}_X(M) : =  \{ m \in M \mid mx = 0, x \in X \} \]
is a submodule of $M$ called the $X$-torsion submodule of $M$ and is denoted as $t_X(M)$. 
It arises as the kernel of the natural map $M \rightarrow MS^{-1}$ where $MS^{-1}$ denotes the corresponding localization of $M$ at $S$. 

Our approach to the determination of the GK dimensions of simple modules will involve the use of localization. To this end we note the following fact.
\begin{lemma}
\label{Ore_sbsts_of_R}
In the notation of Theorem A, the subsets of $R_{n + 1}$ of the form $X: = kH \setminus \{ 0 \}$ where $H \le G_n$ are Ore subsets in $R_n$. 
\end{lemma}
\begin{proof}
This follows from \cite[Exercise 10R]{GW:1989} noting that $X$ is trivially an Ore subset in the commutative domain $K_n = kG_n$. 
\end{proof}

In \cite{BG:2000} C.~J.~B.
~Brookes and J.~R.~J.~ Groves gave an interesting characterization of the GK dimension of finitely generated modules over crossed products of a free abelian group over a division ring. Since the ring $K_n$ is of this type and we shall be making use of this property which we state below. 

\begin{proposition}[Brookes and Groves]\label{crit_GK_dim}
Let $N$ be a finitely generated $K_n = k[x_1^{\pm 1}, \cdots, x_n^{\pm 1}]$-module. Then 
$\gk(N)$ equals the maximal integer $t$ so that $N$ is not torsion as $k[x_1^{\pm 1}, \cdots, x_t^{\pm 1}]$-module.     
\end{proposition}

\begin{proof}
This follows from \cite[Lemma 2.6]{BG:2000} where this number is shown to be equal to the dimension of $N$ in the sense of \cite[Definition 2.2]{BG:2000}. It is also shown in the paragraph following Proposition 4.2 of \cite{BG:2000} that this dimension coincides with the GK dimension.     
\end{proof}

We also recall the definition of a GK-critical module. We shall need this notion for modules over $K_n$. But the definition can be made for an arbitrary affine algebra. 

\begin{definition}
Let $A$ be an affine $k$-algebra and let $M$ be non-zero finitely generated $A$-module. Then $M$ is said to be GK-critical if 
\[ \gk(M/N) < \gk(M) \] holds for each submodule $N < M$ with $N \ne 0$.    
\end{definition}
In particular each simple module is GK-critical. We shall also refer to a GK-critical module as simply critical. The usefulness of working with critical modules owes to the following fact.

\begin{proposition}\label{ctcl_modl_1}
Each nonzero $K_n = kG_n$-module contains a critical submodule.  
\end{proposition}

\begin{proof}
This follows from the more general result which holds for twisted group algebras (c.f. \cite[Proposition 4.3]{GA:2011}).    
\end{proof}

\begin{proposition}\label{ctcl_modl_2}
Let $M$ be a critical $K_n$-module. Then each nonzero submodule $N$ of $M$ is also critical.  
\end{proposition}

\begin{proof}
This assertion follows easily noting Proposition \ref{GK_dim_bsc_ppty_2}.  
\end{proof}

\begin{lemma}\label{no_inf_free_summd}
Let $M$ be a $K_n$-module and let $s$ be maximal with respect to the property that $M$ is not torsion as $K_0: = k[x_1^{\pm 1}. \cdots, x_s^{\pm 1}]$-module. Then $M$ cannot embed a free $K_0$-module of infinite rank.          
\end{lemma}

\begin{proof}
The assertion in the lemma is a special case of \cite[Lemma 2.3]{BG:2000}. 
\end{proof}

\section{Proof of the main theorem}

\begin{mn_thm_1}
Let $K_n$ be the Laurent polynomial algebra \[ K_n : = k[x_1^{\pm 1}, \cdots, x_n^{\pm 1}]  \] over a field $k$ and let \[ R_{n + 1} :=  K_n[x, \delta], \]
where $\delta$ is a $k$ derivation of $K_n$ of the form 
\[\delta = \sum_{i = 1}^n g_i\frac{\partial}{\partial x_i}. \]
for $g_i \in kx_i + k$. 
Suppose that $M$ is a simple $R_n$-module. Then $\gk(M) \in \{1, n\}$.    
\end{mn_thm_1}

\begin{proof}
Set $K: = K_n$ and $R: = R_n$.
Suppose that a simple $R$-module $M$ has GK dimension $m$ such that $2 \le m \le n - 2$. We show below that a contradiction necessarily results. Using Propositions \ref{ctcl_modl_1} and \ref{ctcl_modl_2}, let $N$ be a cyclic and critical $K$-submodule of $M$. If $NR \cong N \otimes_{K} R$ then by lemma \ref{lem_1} below, \[ \gk(M) = \gk(NR) = \gk(N) + 1. \] 
Note that in this case $N$ is a simple $K$-module since $R$ is a flat $K$-module. Hence $\gk(N) = 0$ as each simple $K$-module is finite dimensional (Hilbert Nullstellensatz) and so by the above equation we get $\gk(M) = 1$. 

We now suppose that $M \not \cong N \otimes_K R$. For $d \ge 1$, pick a subset $\mathcal I : = \{i_1, \cdots, i_d\}$ of $\{1, \cdots, n\}$ that is maximal with respect to the property that $N$ is not $\mathcal S$-torsion, where 
$\mathcal S$ is the set of nonzero elements of the subalgebra \[ K' :=   k[x_{i_1}^{\pm 1}, \cdots, x_{i_d}^{\pm 1}] \]  of $K_n$ generated by the $x_i$ and their inverses.

Without loss of generality we may assume $\mathcal I = \{1, \cdots, d \}$. If $d = n$ then $N$ and so $M$ embeds a copy of $K$.  By a well-known fact \[ \gk(K) = n. \]  
It now follows from Proposition \ref{GK_dim_bsc_ppty_1}
that $\gk(M) \ge n$. But on the other hand $\gk(M) \le \gk(R) - 1$ in view of Proposition 5.1(e) of \cite{KL:2000} and so $\gk(M) \le n$. This means that $\gk(M) = n$. In this case also we are done. 

Hence we may suppose that $d < n$. 
Set $\mathcal S : = K' \setminus \{0\}$  
By Lemma \ref{fin_rnk_lem}, we know $M$ does not embed an infinite direct sum of copies of $K'$.  
As noted in Lemma \ref{Ore_sbsts_of_R} $\mathcal  S$ is a right Ore subset in $R$ and we may localize $R$ at $\mathcal S$. The resulting ring $R\mathcal S^{-1}$ which we will denote by $T$ has the form:

\[ T : = R\mathcal S^{-1} = K[x, \delta]\mathcal S^{-1} = Q_d[x_{d + 1}, \cdots, x_n][x, \delta] \]   
where $Q_d$ denotes the fraction field of $K'$.  Note that the derivation $\delta$ of $K$ restricts to a derivation on $K'$ that extends uniquely to a derivation of $Q_d$ in accordance with the quotient rule for ordinary derivatives.  

As $M$ is not $\mathcal S$-torsion the $T$-module $V:= T\mathcal S^{-1} \ne 0$. We claim that $V$ is finite dimensional as $Q_d$-space. Indeed, if this were not true $M$ would necessarily embed a free $K'$-module of infinite rank. But this is contrary to the assertion of Lemma \ref{fin_rnk_lem}.
\end{proof}

\begin{lemma}[GK dimension of induced modules]
\label{lem_1}
Let $N$ be a cyclic $K_n$ sub-module of an $R_{n + 1}$-module $M$. Then 
\[ \gk(N \otimes_{K_n} R_{n + 1}) = \gk(N) + 1. \]
 \end{lemma}

\begin{proof}
We write $R$ for $R_{n + 1}$ and $K$ for $K_n$. Also let $M: = N \otimes_{K} R$. By (\ref{defn_GK_dim_modules}) , \[\gk(M) := \lim_{n \to \infty}\sup\log_n\dim_kN_0R_n,\] where $N_0$ denotes a finite dimensional generating subspace of $N$ (and so also a finite dimensional generating subspace of $M$). 
Let $\{R_m\}_{m \in \mathbb{N}}$ denote the standard finite dimensional filtration of $R$ with respect to the generating set $\{x_1^{\pm 1}, \cdots, x_n^{\pm 1}, x\}.$ 
We can write \[R_m = (\oplus_{i = 0}^{m - 1}  K_ix^{m - i}) \oplus K_m \oplus (\oplus_{i = 0}^{m - 1}  K_ix^{i - m}) \] where 
$\{K_m\}_{m \in \mathbb{N}}$ denotes the standard finite dimensional filtration of $K$ with respect to the generating set $\{x_1^{\pm 1}.
\cdots, x_n^{\pm 1}\}.$ 
Now,  
\begin{align} \label{dim_N_m}
\dim_kN_0R_m &= \dim_k[(\oplus_{i = 0}^{m - 1}  K_ix^{m - i}) \oplus K_m \oplus (\oplus_{i = 0}^{m - 1}  K_ix^{i - m})]  \\ 
&= \displaystyle\sum_{i = 0}^{m - 1} \dim_kN_i + \dim_kN_m,
\end{align}
where $\{N_i\}_{i \in \mathbb{N}}$ stands for the standard filtration of the $K$-module $N$ with respect to the generating subspace $N_0$ and the filtration $\{K_m\}_{m \in \mathbb{N}}$ of $K$. 
Here we are using our assumption that the derivation $\delta$ has the form given in the theorem. This ensures that any word in $x_1^{\pm 1}, x_2^{\pm 1}, \cdots, x_n^{\pm 1}, x$ of length $k$ may be expressed as a linear combination of words of length at most $k$   
in which the only power is $x$ is at the end.   

It is well known that for finitely generated module $N$ over an affine commutative algebra there exists a polynomial $\mathcal H_N$, namely, the Hilbert polynomial  such that for a standard finite dimensional filtration $(N_m)_{m = 0}^\infty$
\[ \dim_k(N_m) = \mathcal H_N(m) \]
for almost all natural numbers $m$. 
In this case it is easily seen that $\gk(M)$ equals the degree $d$ of 
$\mathcal H_N$. In view of (\ref{dim_N_m}) we have, \begin{equation} \label{dim_R_m _L_case}
\dim_k N_m  = 2\displaystyle\sum_{t = 0}^{m}\mathcal H_N(t) - \mathcal H_N(m). 
\end{equation} 
 The sum in the right side of the last equation  
is a polynomial in $m$ of degree $d+1$ by the well-known Faulhaber's formula. It is now clear that \[ \gk(N \otimes_K R) = d + 1. \] 
\end{proof}
\begin{remark}\label{poly_cs_rmk_1}
The foregoing lemma may easily be adapted to the case when the base ring is a polynomial ring $k[x_1, x_2, \cdots, x_n]$. In this case 
\[R_m =   \oplus_{i = 0}^{m}  K_ix^{i - m} \]
and the sum corresponding to 
equation (\ref{dim_R_m _L_case}) will be \[ \dim_k N_m  = \displaystyle\sum_{t = 0}^{m }\mathcal H_N(t) \]
which gives \[ \gk(N \otimes_K R ) = \gk(N) + 1 \] 
by Faulhaber's theorem. 
\end{remark}

\begin{lemma}
\label{fin_rnk_lem}
Suppose that $M$ is a cyclic $R_n$-module with a cyclic and critical $K_n$-module $N$ such that $M = NR_n$.  
Let $t$ be maximal with respect to the property that $N$ embeds a copy of $K': = k[x_1^{\pm 1}, \cdots, x_t^{\pm 1}]$. 
If $NR_n \not \cong N \otimes_{K_n} R_n$ then $M$ cannot embed a free $K'$-module of infinite rank. 
\end{lemma}

\begin{proof}
Set $K := K_n$. We try to mimic the proof of \cite{BG:2000}[Lemma 2.3] which is for the case of a crossed product of a free ablelian group over a division ring.

Since the sum $ \displaystyle\sum_{i = {0}}^\infty Nx^i$ is not direct, we can find a $t$ such that for $r > t$ the $K$-module \[ \displaystyle \sum_{i = 0}^r Nx^i \slash \displaystyle\sum_{i = 0}^{r - 1}Nx^i \] is isomorphic to a proper quotient of $N$ via the obvious map $n \mapsto nx^r$. 

The first part of the hypothesis on $N$ means that $\gk(M) = t$. This follows from Proposition \ref{crit_GK_dim}. 
But $N$ is critical and so each of these quotients is $K^\prime$-torsion. Indeed if this were not true then such a quotient would embed a copy of $K'$ and so by (\ref{GK_dim_bsc_ppty_1}) it would have GK-dimension at least $t$ since $\gk(K^\prime) = t$. But this contradicts the definition of a critical module. Also the sum $\displaystyle\sum_{i = 0}^t n_j \Lambda Y^i$ cannot embed a free $\Lambda^\prime$-submodule of infinite rank by Lemma \ref{no_inf_free_summd}. This proves our claim. 



\end{proof}


\end{document}